\documentclass[11pt]{amsart}
\usepackage{amsmath, amssymb, amsthm}
\usepackage[left=1.5in,top=1.25in,right=1.5in,bottom=1.25in,nohead]{geometry}
\newcommand{\cl}{\operatorname{cl}}

\newcommand{\zncyc}[1]{\left<#1 \right>}

\numberwithin{equation}{section}

\theoremstyle{plain} 
\newtheorem{thm}{Theorem}[section]
\newtheorem{cor}[thm]{Corollary}

\newtheorem{lemma}[thm]{Lemma}
\newtheorem{prop}[thm]{Proposition}

\theoremstyle{definition}

\theoremstyle{remark}

\newtheorem{example}[thm]{Example}

\title{Finite Groups with Many Involutions}	%

\author{Allan L. Edmonds}
\address{Department of Mathematics, Indiana University, Bloomington, IN 47401}
\email{edmonds@indiana.edu}

\author{Zachary B. Norwood}
\address{Department of Mathematics, University of Nebraska, Lincoln, NE 68588}
\email{znorwood@huskers.unl.edu}
\thanks{Research supported in part by NSF grants DMS 0851852 and DMS 0453309
}

\begin{document}
\begin{abstract}It is shown that a finite group in which more than 3/4 of the elements are involutions must be an elementary abelian $2$-group. A group in which exactly 3/4 of the elements are involutions is characterized as the direct product of the dihedral group of order 8 with an elementary abelian $2$-group.
\end{abstract}

\maketitle

\section{Introduction}

It is a standard exercise in an introductory algebra class to show that if $G$ is a group such that $x^{2} = 1$ for all $x \in G$, then $G$ is abelian. It follows that if $G$ is also finite, then it is an elementary abelian $2$-group.

We will show that in fact any finite group in which more than $3/4$ of its elements are involutions must satisfy the same conclusion, that is, it must be an elementary abelian $2$-group. Further, we will also characterize a group for which the proportion of involutions is exactly $3/4$ as the direct product of a dihedral group of order 8 with an elementary abelian 2-group.

Although none of the results in this paper depend on computer calculation, explorations using the GAP program \cite{GAP2005} were important in formulating the results.

\section{Definitions and Examples}

A good reference for basic group theory is Aschbacher's text \cite{Aschbacher2000}. In particular, Section 45 contains results on the number of involutions in a finite group and additional references to the literature.

Let $G$ denote a finite group, written multiplicatively with identity element 1. We refer to any element $x\in G$ such that $x^{2}=1$, including the identity element, as an \emph{involution}. We write $J(X)$ for the set of involutions in a subset $X$ of a group, and write $j(X)$ for its cardinality $|J(X)|$. We also find it convenient to define the invariant $\alpha(G)=j(G)/|G|$ representing the proportion of involutions in the group $G$.
Note that $\alpha(G)\in (0,1]$.

\begin{prop}			\label{prop:abelian}
	If $G$ is a finite abelian group, then $J(G)$ is an elementary abelian subgroup of $G$, and $j(G)$ is a power of $2$, dividing $|G|$.
\end{prop}
\begin{proof}
One easily verifies that $J(G)$ is closed under inversion in general and under the group operation when $G$ is abelian. It is then a subgroup. Since every element has order $2$, it is an elementary abelian $2$-group, and its order, $j(G)$ is therefore a power of $2$ dividing $|G|$.
\end{proof}

\begin{cor}
If $G$ is a finite abelian group and $\alpha(G)>1/2$, then $G$ is an elementary abelian $2$-group.\qed
\end{cor}
In general, of course, $J(G)$ is not a subgroup of $G$. Dihedral groups form the most important class of non-abelian groups with many involutions.
\begin{example}
If $G=D_{2n}$, the dihedral group of order $2n$, then
\[
j(G)=
\begin{cases}
n+1 &\text{ if } n \text{ is odd}\\
n+2 &\text{ if } n \text{ is even}
\end{cases}.
\]
\end{example}
In particular 
\[\alpha(D_{2n})=
\begin{cases}
\frac{1}{2}+\frac{1}{2n} &\text{ if } n \text{ is odd}\\
\frac{1}{2}+\frac{1}{n} &\text{ if } n \text{ is even}
\end{cases}
\]
and $\alpha(D_{2n})\le 3/4$, unless $n=2$ and the dihedral group is actually elementary abelian.

\section{Preliminary Results}
We record some basic, useful, facts about counting involutions in finite groups.
\begin{lemma}
If $G=H\times K$, then $J(G)=J(H)\times J(K)$, $j(G)=j(H)\times j(K)$, and $\alpha(G)=\alpha(H)\times \alpha(K)$.
\end{lemma}
\begin{proof}
One simply observes that a pair $(h,k)$ in $G$ is an involution if and only if both $h$ and $k$ are involutions.
\end{proof}

\begin{lemma}
	If $G$ is a finite group with a normal subgroup $H$, then $j(G) \leq |H| \times j(G/H)$, and
	$\alpha(G)\le\alpha(G/H)$.
\end{lemma}

\begin{proof}
	Clearly each involution of $G$ maps to an involution (perhaps trivial) in $G/H$, and over any involution of $G/H$ there are at most $|H|$ involutions in $G$. The result follows.
\end{proof}

\begin{lemma}\label{lemma:centralsubgroup}
		If $G$ is a finite group with a  central subgroup $H$, then $j(G) \leq j(G/H)j(H)$, and
	$\alpha(G)\le \alpha(G/H)\alpha(H)$.
\end{lemma}
\begin{proof}
	As before, each involution $x\in G$ maps to an involution $\bar{x}\in G/H$. Over the involution $\bar{x}$ of $G/H$ are the elements of the form $xh$, $h\in H$. Since $H$ is central, an element of the form $xh$ is an involution if and only if $h$ is an involution. The result follows.
\end{proof}

\begin{lemma}
If $G$ is a group expressed as a semidirect product $NQ\cong N\rtimes Q$, then
\[
J(G)=\{nq: n\in N, q\in Q, q^{2}=1, qnq=n^{-1}\}
\]
\end{lemma}
\begin{proof}
The proof is an easy calculation: If $q^{2}=1$ and $q$ inverts $n$, then $(nq)^{2}=nqnq=nqnq^{-1}=nn^{-1}=1$. Conversely, if $nq$ is an involution, then so is its image $q$ in the quotient group $Q$. Therefore $1=(nq)^{2}=nqnq^{-1}$, implying that $qnq^{-1}=n^{-1}$. 
\end{proof}

Note that $G=N$ is the disjoint union of the cosets $Nq$, $q\in Q$, and a coset $Nq$ contains involutions if and only if $q$ is an involution, and the involutions in $Nq$ are in one-to-one correspondence with the elements of $N$ that are inverted by the action of $q$ on $N$ by conjugation.

\begin{prop}
If $G$ is a finite group, $S< G$ a Sylow $2$-subgroup with normalizer $N=N_{G}(S)$, then $\alpha(G)\le|S|/|N|$.
\end{prop}
\begin{proof}
Every involution lies in some Sylow $2$-subgroup and all Sylow $2$-subgroups are conjugate. Therefore $J(G)\subset \bigcup_{g\in G}gSg^{-1}$. We need only take the union over a set of coset representatives of $N$ in $G$. It follows that
\[
j(G)\le \left|G/N\right|\left(|S|-1\right)+1\le |G/N|\times |S|=|G|\, |S|/|N|
\]
Dividing through by $|G|$, the result follows.
\end{proof}

\begin{cor}[\cite{Edmonds2009}, Corollary 4.4]\label{cor:normalizer}
If $G$ is a finite group with $\alpha(G)>1/2$ and $S$ is a Sylow $2$-subgroup, then $N_{G}(S)=S$.\qed
\end{cor}

\begin{prop}\label{prop:center}
	If $G$ is a finite group such that $j(G) > |G|/2$, then the center $Z(G)$ is an  elementary abelian 2-group.
\end{prop}

The example of a dihedral group $D_{2n}$, $n$ odd, shows that the center $Z(G)$ of such a group with $\alpha > 1/2$ could be trivial.

\begin{proof}
	Let $Z = Z(G)$ and let $S$ be a Sylow 2-subgroup of $G$.
	By Corollary \ref{cor:normalizer}, $N_{G}(S) = S$, so
		$Z \leqslant S$. Therefore $Z$ is a 2-group; let $|Z| = 2^{a}$
		 and $j(Z) = 2^{b}$ for some $0 \leq b \leq a$.
	Now we have
	\begin{align*}
		{|G|}/{2} & < j(G)\\
				 & \leq j(G/Z)j(Z) \quad\quad \text{(by Lemma \ref{lemma:centralsubgroup})}\\
				 & = j(G/Z)2^{b}\\
				 & \leq \left({|G|}/{|Z|}\right)2^{b}.
	\end{align*}
	It follows that  $|Z| < 2^{b+1}$. But $Z$ is a 2-group of order at least $j(Z) = 2^{b}$, so $|Z| = j(Z)$ and $Z = J(Z)$. Since every element of $Z$ is an involution, the result follows.
\end{proof}

Note that the inequality in the hypothesis of the proposition must be strict:\\
Consider $G = C_{4} \times (C_{2})^{n-2}$. The order of $G$ is $2^{n}$, and $j(G) = 2^{n-1} = |G|/2$, but $Z(G) = G$, which is not an elementary abelian 2-group.

We recall a crucial result for handling non-2-groups, which provided the starting point for the present investigation.

\begin{thm}[Edmonds \cite{Edmonds2009}, Theorem 4.1]\label{thm:edmonds}
	If $|G| = 2^{n}m$, $m$ odd and $n\ge 1$, then $j(G) \leq 2^{n-1}(m+1)$.\qed
\end{thm}

\begin{cor}\label{cor:twothirds}
If $|G| = 2^{n}m$, $m$ odd and $n\ge 1$, then $\alpha(G)\le \frac{m+1}{2m}=\frac{1}{2}+\frac{1}{2m}$. If $m>1$, then $\alpha(G)\le \frac{2}{3}$.\qed
\end{cor}

Edmonds \cite{Edmonds2009} also proved that a finite group with $j(G) =2^{n-1}(m+1)$ is the direct product of ${C_{2}}^{n-1}$ and a group of order $2m$ of dihedral type, i.e., a split extension of an abelian group of order $m$ by the cyclic group of order 2 acting by inversion.

\section{Groups with $\alpha > 3/4$}
Here we present our characterization of groups in which more than $3/4$ of the elements are involutions.
\begin{thm}			\label{thm:main}
	If $G$ is a finite group and $\alpha(G) > 3/4$, then $G$ is an elementary abelian
	$2$-group.
\end{thm}
\begin{proof}
	By Corollary \ref{cor:twothirds}, $G$ must be 
	of order $2^{n}$ for some $n$. We proceed by induction on $n$.
	We may suppose $n > 3$, as the case when $n \leq 3$ follows easily by inspection of the lists of groups of small order, at most $8$.
	Because $G$ is a $2$-group, its center is nontrivial, and so contains an involution.
	Let $a$ be a central involution in $G$. By Lemma \ref{lemma:centralsubgroup}, we have
	$j(G) \leq j(G/\!\zncyc{a})j(\zncyc{a})$; hence 
		$$j(G/\!\zncyc{a}) \geq {j(G)}/{2} > {3}|G|/8 = {3}|G/\!\zncyc{a}\!|/4$$
	Therefore, by the inductive hypothesis, $G/\!\zncyc{a}$ is an elementary abelian $2$-group
	of order $2^{n-1}$. So we have a central extension
		$$C_{2} \rightarrowtail G \twoheadrightarrow (C_{2})^{n-1}$$
	which we must show to be a direct product. It is a direct product if and only $G$ is elementary abelian.	
	
	\indent If $G$ is not an elementary abelian $2$-group, there must be an element of order 4.
	Suppose $G$ contains an element of order 4; call it $x$. Then, under the natural projection
	onto the quotient group $G/\!\zncyc{a}$, $\zncyc{x}$ is the 
	inverse image of a
	cyclic subgroup of order 2. Because that cyclic subgroup must be normal in the abelian quotient, 
	$\zncyc{x} \vartriangleleft G$.
	Consequently, conjugation by an element of $G$ must send $x$ to either $x$ or $x^{-1}$,
	since those are the only two candidates inside $\zncyc{x}$.
	In particular, the length of the orbit of $x$ under conjugation is at most 2.
	In fact, since by Proposition \ref{prop:center} there are no elements of order 4 in $Z(G)$, 
		$$|\cl(x)| = [G : C(x)] = 2,$$
	where $\cl(x)$ denotes the conjugacy class of $x$, and $C(x)$ denotes its centralizer.
	Now note that
		$$j(C(x)) = j(G) - j(G-C(x)) > {3}|G|/{4} - |G|/2 = |G|/4= |C(x)|/2.$$
	We can therefore apply Proposition \ref{prop:center} to $C(x)$.
	It follows that $Z(C(x))$ is an elementary abelian 2-group, contradicting the fact that
	 $x$, an element of order 4, is necessarily in the center of its own centralizer.
So $G$ contains no element of order 4 and is thus an elementary abelian 2-group,
	as required. This completes the inductive step and hence the proof of the theorem.
\end{proof}

The inequality in the hypothesis of the theorem is the best possible, as the results in the following section demonstrate.

\section{Groups with $\alpha=3/4$}
The simplest group with $\alpha=3/4$ is the dihedral group $D_{8}$ of order 8. We will show here that a finite group $G$ with $\alpha(G)=3/4$ is of the form $D_{8}\times {C_{2}}^{k}$.

\begin{lemma}\label{lemma:surjection}
Suppose $G$ is a finite group of order $2^{n}$ such that $\alpha(G)=3/4$, and suppose there is a surjection $\pi:G\to D_{8}$. Then $K=\ker \pi\cong {C_{2}}^{n-3}$, the surjection $\pi$ splits, so that $G$ is a semidirect product $K\rtimes D_{8}$, and the semidirect product is in fact a direct product.
\end{lemma}
\begin{proof}
Each involution of $G$ maps to one of the six involutions of $D_{8}$. Over each one of these involutions there are at most $|K|=2^{n-3}$ involutions. In order to reach $\alpha(G)=3/4$ it is necessary that all elements in the preimage of any of the involutions of $D_{8}$ must be involutions.  We conclude that $K$ consists entirely of involutions, so that $K\cong {C_{2}}^{n-3}$. Choose two involutions $x,y\in G$, mapping to two non-commuting involutions in $D_{8}$, which necessarily generate $D_{8}$. Then in $G$ the subgroup $\left<x,y\right>$ is a dihedral group, which $\pi$ maps onto $D_{8}$. Now $(xy)^{2}$ maps to the central involution of $D_{8}$. Therefore $(xy)^{2}$ is also an involution in $G$. We conclude that $\left<x,y\right>$ defines a copy of $D_{8}$, expressing $G$ as a semidirect product $K\rtimes D_{8}$. Now an element $(a,b)$ of such a semidirect product is an involution if and only if $b$ is an involution and $b$ conjugates $a$ to its inverse. Since $K$ is an elementary abelian $2$-group, $a^{-1}=a$, so $(a,b)$ is an involution if and only if $b$ commutes with $a$. And for every involution $b\in D_{8}$ we must have $(a,b)$ an involution for all $a\in K$. We conclude that all involutions in $D_{8}$ act trivially on $K$. Since $D_{8}$ is generated by involutions, the entire group $D_{8}$ acts trivially on $K$ and the semidirect product is a direct product, as required.
\end{proof}

\begin{thm}
Let $G$ be a finite group such that $\alpha(G)=3/4$. Then $|G|=2^{n}$, $n\ge 3$, and $G\cong D_{8}\times {C_{2}}^{n-3}$.
\end{thm}
\begin{proof}By Corollary \ref{cor:twothirds} it follows  that $G$ is a nonabelian $2$-group, of order $2^{n}$, say, where $n\ge 3$. We therefore proceed by induction on $n$. When $n=3$ the result follows from the elementary classification of finite groups of order 8.

Let $Z=Z(G)$ be the nontrivial center of $G$, which by Proposition \ref{prop:center} is known to be an elementary abelian 2-group ${C_{2}}^{r}$. Consider the quotient group $Q=G/Z$. By Lemma \ref{lemma:centralsubgroup} we know that $3|G|/4=j(G)\le j(Q)|Z|$. It follows that $j(Q)\ge 3|Q|/4$. If $j(Q)=3|Q|/4$, then by induction on order, we may assume that $Q\cong D_{8}\times {C_{2}}^{t}$, for some $t$. In this case $G$ clearly admits a surjection onto $D_{8}$, so Lemma \ref{lemma:surjection} implies that $G$ has the required properties.

If  $j(Q)\ne 3|Q|/4$, then  $j(Q)>3|Q|/4$, and $Q\cong {C_{2}}^{s}$ where $r+s=n$, by Theorem \ref{thm:main}. 
We have a central extension
\[
{{C_{2}}^{r}} \rightarrowtail G \stackrel{\pi}{\twoheadrightarrow} {C_{2}}^{s}
\]
where ${{C_{2}}^{r}}$ is the full center of $G$, and $\pi\colon G\to {C_{2}}^{s}$ is the projection map.

Over any involution in $Q$ there is either a full set of $2^{r}$ involutions (all obtained by multiplying one such involution by an element of $Z$) or there are no involutions. Thus $j(G)=|\pi (J(G))|\times 2^{r}$ and $|\pi (J(G))|=3\times 2^{s-2}$.

Since $G$ is non-abelian and generated by $J(G)$, there must be two non-commuting involutions  $x,y\in G$, which map nontrivially to $\bar{x},\bar{y}\in Q$, such that $(xy)^{2}\ne 1$. In particular, $\bar{x}\bar{y}\ne 1$. Then $a=(xy)^{2}$ is an element of $Z$. In particular $\bar{x}\bar{y}$ is an involution of $Q$ that is not the image of an involution of $G$. 

Then $\left<x,y\right>\cong D_{8}$. 
and  ${D_{8}}\cap Z=\left<a\right>$ for the central involution $a=(xy)^{2}$. Note that $a$ is also a commutator $[x,y]$. Extend $\{a\}$ to a basis $a,f_{2},\dots f_{r}$ of $Z$,  as a vector space over the field of two elements, and let $Z'=\left< f_{2},\dots f_{r}\right>$.

Suppose $Z'$ is nontrivial. As a subgroup of the center, $Z'$ is normal in $G$. Consider the quotient group $R=G/Z'$. Note that $R$ is not abelian since the commutator $a$ is not in $Z'$. Moreover, as before, $\alpha(R)\ge 3/4$. Since $R$ is not abelian, we cannot have $\alpha(R)>3/4$, by Theorem \ref{thm:main}. Therefore $\alpha(R)=3/4$. Since $Z'$ is nontrivial, $|R|<|G|$, so that by induction on order the group  $R$ can be expressed as a direct product of $D_{8}$ and an elementary abelian 2-group. In particular, $R$, and hence $G$, maps onto $D_{8}$. By Lemma \ref{lemma:surjection} $G$ itself can be expressed as a direct product of $D_{8}$ and an elementary abelian 2-group.

It remains to consider the case where $Z=\left<a\right>$, and $R=G/Z\cong {C_{2}}^{n-1}$. We will show that under this assumption we necessarily have $n=3$ and $G=D_{8}$. As above we have the two involutions $x,y\in G$ that generate a $D_{8}$ and whose images $\bar{x},\bar{y}\in R$ generate a summand ${C_{2}}^{2}$ of $R$. We aim to show that this $D_{8}$ is the whole group.

Now $D_{8}$ is normal in $G$ with quotient $S$ isomorphic to ${C_{2}}^{n-3}$. Let $\pi:G\to S$ be the quotient map. For any involution $t\in G$ such that $\pi(t)=\bar{t}$ is nontrivial in $S$, the subgroup $\left<x,y,t\right>$ in $G$ is a semidirect product $D_{8}\rtimes C_{2}$, where $t\in C_{2}$ acts on $D_{8}$ by conjugation.

 The automorphism group of $D_{8}$ is ``well-known.''  Compare Hall \cite{Hall1958}, Exercise 1, page 90. It is abstractly isomorphic to $D_{8}$, although, of course, not every automorphism is an inner automorphism. Among the 8 automorphisms there are 6 involutions. These involutions each invert at most 6 elements of $D_{8}$, since $D_{8}$ is nonabelian. Three of them, including the identity, invert 6 elements.

 It follows that an element of $S$ is hit by up to six involutions or by none. Now $\frac{3}{4}|G|=j(G)\le |\pi (J(G))|\times 6\le 6\times 2^{n-3}$. Therefore  the inequality is an equality and $\pi (J(G))=S$. That is, each element of $S$ is hit by exactly 6 involutions of $G$. 

We conclude that each such coset $\left<x,y\right>t$ contains exactly 6 involutions and that for each element of $S$ there is an involution $t$ that maps to it. A coset  $\left<x,y\right>t$ contains involutions if and only if $t$ is an involution. Therefore the set-theoretic difference $G-\left<x,y\right>$ consists entirely of involutions, each of which acts on $D_{8}$ as one of the involution automorphisms that inverts 6 elements.

In particular $G-D_{8}$ consists entirely of involutions. But then
\[
J(G)=J(D_{8})\cup (G-D_{8})
\]
and
\[
(3/4)\times 2^{n}=6+(2^{n}-8)
\]
It follows that $n=3$ and in this case we already know the result. This completes the proof.
 \end{proof}



\begin{thebibliography}{10}
\bibitem{Aschbacher2000}Michael Aschbacher, \textit{Finite Group Theory}, Second edition,Cambridge University Press, 2000.
	\bibitem{Edmonds2009} Allan L. Edmonds, \textit{The partition problem for equifacetal simplices}, Contributions to Algebra and Geometry, \textbf{50} (2009), 195-213.
 \bibitem{GAP2005}
  The GAP~Group, \textit{GAP -- Groups, Algorithms, and Programming, 
  Version 4.4}; 2005,
\texttt{www.gap-system.org}.
\bibitem{Hall1958}Marshall Hall, Jr.,
\textit{The Theory of Groups} ,The Macmillan Co., New York, N.Y. 1959. Reprinted by the Chelsea Publishing Co., New York, 1976.
\end{thebibliography}
\end{document}